\documentclass[12pt,centertags,oneside]{amsart}
\usepackage{amsmath,amstext,amsthm,amscd,typearea,hyperref}
\usepackage{amssymb}
\usepackage{a4wide}
\usepackage[mathscr]{eucal}
\usepackage{mathrsfs}
\usepackage{typearea}
\usepackage{charter}
\usepackage{pdfsync}
\usepackage{stmaryrd}
\usepackage[T1]{fontenc}
\usepackage{enumitem}
\usepackage{dsfont}

\usepackage{tikz-cd}
\usepackage{pgfplots}
\tikzset{>=latex}
\pgfplotsset{compat=newest}
\usepackage{caption}
\usepackage{adjustbox}

\usepackage[a4paper,width=16.5cm,top=3cm,bottom=3cm]{geometry}

\numberwithin{equation}{section}

\usepackage[french,english]{babel}

\allowdisplaybreaks
\emergencystretch=\maxdimen
\usepackage{multicol}

\usepackage{xcolor}

\newtheorem{theorem}{Theorem}[section]
\newtheorem{definition}[theorem]{Definition}
\newtheorem{proposition}[theorem]{Proposition}

\newtheorem{lemma}[theorem]{Lemma}
\newtheorem{remark}[theorem]{Remark}

\newtheorem*{definition*}{Definition}

\newcommand{\cali}[1]{\mathscr{#1}}

\newcommand{\supp}{{\rm supp}}

\newcommand{\diff}{{\rm d}}

\newcommand{\DSH}{{\rm DSH}}

\newcommand{\del}{\partial}

\newcommand{\ddc}{{\rm dd^c}}

\newcommand{\dd}{{\rm d}}

\newcommand{\dbar}{\overline\partial}
\newcommand{\ddbar}{\partial\overline\partial}

\newcommand{\sign}{{\rm sign\ \!}}

\newcommand{\Cc}{\cali{C}}

\newcommand{\Fc}{\cali{F}}

\newcommand{\Qc}{\cali{Q}}

\newcommand{\cA}{\mathcal{A}}
\newcommand{\cB}{\mathcal{B}}

\newcommand{\FS}{{\rm FS}}
\newcommand{\B}{\mathbb{B}}

\newcommand{\C}{\mathbb{C}}

\newcommand{\N}{\mathbb{N}}
\newcommand{\Z}{\mathbb{Z}}
\newcommand{\R}{\mathbb{R}}

\renewcommand\P{\mathbb{P}}

\newcommand{\lp}{\langle}
\newcommand{\rp}{\rangle}

\newcommand{\norm}[1]{\lVert#1\rVert}

\newcommand{\oA}{\mathcal{A}}
\newcommand{\oB}{\mathcal{B}}

\newcommand{\oN}{\mathcal{N}}

\newcommand{\oE}{\mathcal{E}}

\newcommand{\bv}{\mathbf{v}}



\title{Mixing and CLT for H\'enon-Sibony maps: plurisubharmonic observables}

\author{Marco Vergamini}
\address{Scuola Normale Superiore, Pisa, Italy}
\email{marco.vergamini@sns.it}

\author{Hao Wu}
\address{National University of Singapore}
\email{e0011551@u.nus.edu}

\thanks{}

\begin{document}

\begin{abstract}
Let $f$ be a complex H\'enon map and $\mu$ its unique measure of maximal entropy. We prove that $\mu$ is exponentially mixing of all orders for all (not necessarily bounded) plurisubharmonic observables, and that all plurisubharmonic functions satisfy the central limit theorem with respect to $\mu$. Our results hold more generally for every H\'enon-Sibony map on $\mathbb{C}^k$.
\end{abstract}

\clearpage\maketitle
\thispagestyle{empty}


\medskip


\setcounter{tocdepth}{1}

\medskip

\section{Introduction} \label{intro}

H\'enon maps are among the most studied dynamical systems. They were introduced in the real setting by Michel H\'enon \cite{henon}. Since the 1980s, the study of these maps in the complex setting has seen a surge in activity, following the development of pluripotential theory techniques, see for instance the works of Bedford-Lyubich-Smillie \cite{bed-1993,bedford1,bed-1991,bedford2} and Fornaess-Sibony \cite{fornaess:survey,forn,sibony:panorama}. In particular, a unique measure of maximal entropy $\mu$ (also called \textit{equilibrium measure}) for these systems was introduced by Sibony, as the intersection of two canonical positive closed currents, and the study of the ergodic properties of $\mu$ has quickly become a central problem in the domain. It is shown in \cite{bedford1} that $\mu$ is Bernoulli. In particular, it is mixing of all orders. On the other hand, the control of the speed of mixing is a challenging problem, and is usually obtained only under strong hyperbolicity assumptions on the systems, see for instance \cite{liverani}.

\smallskip

In \cite{bian-dinh-sigma}, Bianchi-Dinh proved the exponential mixing of all orders and the central limit theorem for \textbf{smooth} (and, by the interpolation techniques, \textbf{H\"older continuous}) observables for Hénon maps on $\mathbb{C}^2$ and, more generally, for all so called \emph{Hénon-Sibony maps} on $\mathbb{C}^k$ \cite{deth-auto,rigidity,DS-saddle}. This solved a long-standing question in the domain. De Th\'elin-Vigny \cite{hen-gab-CLT} recently generalized this result for observables in a special subspace of \textbf{bounded d.s.h.\ functions} (their result is valid in the more general setting of generic birational maps \cite{bedford-diller,dtv-memoir}). The goal of this paper is to obtain these strong statistical properties for all \textbf{d.s.h.\ observables}. We recall that, roughly speaking, d.s.h.\ functions are the elements of the vector space spanned by (non necessarily bounded) plurisubharmonic (p.s.h.\ for short) functions, see \cite{din-sib-cmh} and Section \ref{defs} for precise definitions. Since their introduction, d.s.h.\ functions have been proven to be a fundamental tool in the study of holomorphic dynamical systems; they are naturally adapted to the underlying complex structure, and invariant by iteration of holomorphic maps, something that is not true for the usual ``real'' regularity conditions. We refer to \cite{bd-gafa,dinh-exponential,dinh-sibony:acta,dinh-sibony:cime,gov,gauthier-vigny} for examples of applications.

\smallskip

Let us recall the following definition.

\begin{definition}\label{def-mix}\rm
Let $(X,f)$ be a dynamical system and $\nu$ an $f$-invariant measure. Let $(\oE,\norm{\cdot}_\oE)$ be a normed space of real functions on $X$ with $\norm{\cdot}_{L^q(\nu)}\lesssim_q\norm{\cdot}_\oE$ for all $1 \le q < +\infty$. We say that $\nu$ is \textit{exponentially mixing of all orders} for observables in $\oE$ if, for every $\kappa\in\mathbb{N^*}$, there exist constants $C_\kappa > 0$ and $0 < \theta_\kappa < 1$ such that, for all $\varphi_0, \dots, \varphi_\kappa$ in $\oE$ and integers $0=:n_0 \le n_1 \le \ldots \le n_\kappa$, we have
$$\Big| \int \varphi_0 (\varphi_1\circ f^{n_1})\cdots (  \varphi_\kappa \circ f^{n_\kappa}) \,\dd \nu -\prod_{j=0}^\kappa  \int \varphi_j \,\dd \nu \Big| \leq C_\kappa  \,\theta_\kappa^{ \min_{0\leq j\leq \kappa-1}(n_{j+1}-n_j)  }\prod_{j=0}^\kappa \norm{\varphi_j}_\oE.$$
\end{definition}

The following is our first main result, which completely settles the question of the exponential mixing for d.s.h.\ observables.

\begin{theorem} \label{thm-mixing}
Let $f$ be a Hénon-Sibony map and $\mu$ its equilibrium measure. Then, $\mu$ is exponentially mixing of all orders for all observables in $\DSH(\mathbb{P}^k)$. More precisely: there exists $0<\theta<1$ such that for every integers $\kappa\in \N^*$,  $0=n_0\leq n_1\leq \ldots\leq n_\kappa$ and for every $\varphi_0,\varphi_1,\dots, \varphi_\kappa \in \DSH(\P^k)$, we have
$$\Big| \int \varphi_0 (\varphi_1\circ f^{n_1})\cdots (  \varphi_\kappa \circ f^{n_\kappa}) \,\dd \mu -\prod_{j=0}^\kappa  \int \varphi_j \,\dd \mu \Big| \leq C_\kappa  \,\theta ^{ \min_{0\leq j\leq \kappa-1}(n_{j+1}-n_j)  }\prod_{j=0}^\kappa \norm{\varphi_j}_\DSH,$$
where  $C_\kappa>0$ is a constant independent of $n_1,\dots,n_\kappa,\varphi_0,\dots,\varphi_\kappa$.
\end{theorem}

The exponential mixing of order 1 (i.e., the case of 2 observables) was first proved by Dinh for $\Cc^\alpha$ ($0<\alpha\le2$) observables \cite{dinh-decay-henon}. The second author extended this result to d.s.h.\ observables \cite{Wu-Ergodic}. For generic birational maps, the exponential mixing of order 1 for $\Cc^\alpha$ observables was shown by Vigny \cite{vigny-decay}. Although it is known that the exponential mixing of order 1 often implies that the measure is Bernoulli \cite{dolgopyat}, giving as a consequence the mixing of all orders, these results do not say anything about the rate of convergence, see for instance \cite[Question 1.5]{dolgopyat}.

\medskip

In \cite{bian-dinh-sigma,dinh-decay-henon}, the problem for H\"older observables can be reduced to $\Cc^2$ observables by interpolation \cite{triebel}. Moreover, as every $\Cc^2$ function is the difference of two p.s.h.\ functions, the observables $\varphi_j$'s can be assumed to be p.s.h. near the support of $\mu$.
A key idea is then to consider the new dynamical system $$F(z,w):=\big(f(z), f^{-1}(w)\big)$$ (which, when $f$ is a H\'enon map on $\C^2$, is a H\'enon-Sibony map on $\C^4$), and the new observables 
$$\big(\varphi_0(w)+M\big)\cdot \prod_{j=1}^\kappa \big(\varphi_j\circ f^{n_j-n_1}(z)+M\big)$$
and
$$\big(-\varphi_0(w)+M\big)\cdot \Big(\prod_{j=1}^\kappa \big(\varphi_j\circ f^{n_j-n_1}(z)+M\big) -2(M+1)^\kappa\Big),$$
which are p.s.h.\ for $M$ large enough. Thus, an estimate on iterations of positive closed currents (see \cite[Corollary 2.5]{bian-dinh-sigma} and Lemma \ref{lem-key} below) can be applied. However, the two functions above may not be p.s.h.\ anymore if we remove the $\Cc^2$ assumption on the $\varphi_j$'s. Hence the same  strategy does not work in the general setting of Theorem \ref{thm-mixing}. A similar problem appears when trying to adapt the proof of \cite{hen-gab-CLT}, which relies on a similar decomposition. Our proof relies both on a stronger mixing result than that of \cite{hen-gab-CLT} in the case of bounded p.s.h.\ observables, which we prove in Proposition \ref{prop-mixing-bounded} (see also Remark \ref{somewhat-better}), and on a precise control of the singularities of d.s.h.\ functions, combined with a precise regularity estimate for $\mu$.

\medskip

Our second result is the central limit theorem for all d.s.h.\ observables. Fix an observable $\varphi\in \DSH(\P^k)$, and set 
$$ S_n(\varphi):=  \varphi +\varphi \circ f +\cdots +\varphi \circ f^{n-1}.  $$
Define the non-negative constant $\sigma=\sigma(\varphi)$ by 
$$ \sigma^2 :=  \lim_{n\to \infty} \frac{1}{n} \int \big(S_n(\varphi) -  \lp \mu, \varphi\rp \big)^2 \,\dd \mu.$$\\
By Birkhoff’s ergodic theorem, we have
$$ n^{-1} S_n (\varphi) (x) \longrightarrow \lp \mu,\varphi \rp  \quad \text{for $\mu$-almost every } x\in \P^k.$$
This is an analogue of the strong law of large numbers. Our second main result is the following control of the rate in the above convergence.
We let $\oN(0,\sigma^2)$ denote the Gaussian distribution with mean $0$ and variance $\sigma^2$ (when $\sigma=0$, we mean that $\oN(0,\sigma^2)$ is the trivial point distribution at $0$).

\begin{theorem}\label{thm-clt}
Let $f$ and $\mu$ be as in Theorem \ref{thm-mixing}. Then, every $\varphi\in\DSH(\P^k)$ satisfies the central limit theorem with respect to $\mu$. Namely, we have
\begin{equation} \label{law-conv}
    \frac{S_n(\varphi )-  n\lp \mu, \varphi \rp}{\sqrt n}  \longrightarrow  \oN(0,\sigma^2)  \quad \text{as $n\to\infty$ in law}.
\end{equation}
\end{theorem}

When $\sigma\neq 0$, this means that for any interval $I\subset \R$ we have
$$\lim_{n\to\infty} \mu \Big\{ \frac{S_n(\varphi )-  n\lp \mu, \varphi \rp}{\sqrt n}  \in I  \Big\} =  \frac{1}{\sqrt{2\pi}\,\sigma}\int_I e^{-t^2/(2\sigma^2)} \,\dd t.$$

For $\varphi \in \Cc^{\alpha}(\P^k),0<\alpha\leq 2$, the above central limit theorem was deduced by Bianchi-Dinh \cite{bian-dinh-sigma} from their version of Theorem \ref{thm-mixing} for $\Cc^\alpha$ observables. The proof relies on a theorem by Bj\"orklund-Gorodnik \cite{bjo-gor-clt}, which states that the exponential mixing of all orders for observables bounded in a norm $\norm{\cdot}_\oE$ implies the central limit theorem, under some natural conditions including
\begin{equation} \label{submul}
    \norm{\varphi_1\varphi_2}_\oE\leq c \norm{\varphi_1}_\oE \norm{\varphi_2}_\oE\quad\text{for every}\quad\varphi_1,\varphi_2\in\oE.
\end{equation}
We cannot directly apply \cite{bjo-gor-clt} in our case, because $\varphi_1 \varphi_2$ is not in general d.s.h.\ for $\varphi_1,\varphi_2\in \DSH(\P^k)$. So, \eqref{submul} clearly does not hold in our setting. We will prove a more general version of \cite[Theorem 1.1]{bjo-gor-clt}, see Theorem \ref{general-clt}, from which, together with Theorem \ref{thm-mixing}, we will deduce Theorem \ref{thm-clt}.

\medskip

\noindent\textbf{Notations.}
The symbols $\lesssim$ and $\gtrsim$ stand for inequalities up to a positive multiplicative constant, and a subscript means that said constant can depend on some variables, e.g., $\lesssim_t$ means that the implicit constant can depend on the variable $t$. We denote by $\B_k(x,r)\subset \C^k$ the open ball centered at $x\in \C^k$ with radius $r>0$, and by $\overline{\B}_k(x,r)$ the closed one. The pairing $\lp \cdot,\cdot \rp$ is used for the integral of a function with respect to a measure or, more
generally, the value of a current at a test form. The mass of a positive closed $(p,p)$-current $S$ on $\mathbb{P}^k$ is defined as $\|S\|:=\lp S,\omega_{\FS}^{k-p}\rp$, where $\omega_{\FS}$ denotes the Fubini-Study form.

\medskip

\section{Preliminaries} \label{defs}

\subsection{Pluripotential theory}

Recall that a function $\varphi:\mathbb{P}^k\rightarrow\R \cup \{-\infty\}$ is called \textit{quasi-plurisubharmonic}  (\textit{quasi-p.s.h.}\ for short) if it is locally the difference of a p.s.h.\ function and a smooth one. The space of these functions is larger than that of p.s.h.\ functions, which are necessarily constant by the maximum principle. On the other hand, quasi-p.s.h.\ functions still preserve the main compactness properties of p.s.h.\ functions, some of which we will review in this subsection. A function $\varphi:\mathbb{P}^k\rightarrow\R \cup\{\pm \infty\}$ is \textit{d.s.h.}\ if it is the difference of two quasi-p.s.h.\ functions outside of a pluripolar set. These functions  have played  a central role in holomorphic dynamics. Denote by $\DSH(\P^k)$ the space of d.s.h.\ functions on $\P^k$. If $\varphi$ is d.s.h., there are two positive closed $(1,1)$-currents $R^\pm$ on $\P^k$ such that $\ddc \varphi =R^+ -R^-$. As these two currents are cohomologous, they have the same mass. We define a norm on $\DSH(\P^k)$ by 
$$\norm{\varphi}_\DSH:= \Big| \int  \varphi \,\omega_\FS^k  \Big|+\inf \norm {R^\pm},$$
where the infimum is taken over all $R^\pm$ as above. We obtain an equivalent norm if, instead of $\omega_\FS^k$, we take any measure $\nu$ that is \emph{PB}, i.e., such all that d.s.h.\ functions are integrable with respect to $\nu$.
We will need the following decomposition result for d.s.h.\ functions.

\begin{lemma} \label{dsh-split-psh}
Let $B$ be a bounded domain in the chart of $\mathbb{P}^k$ canonically identified with $\mathbb{C}^k$, and let $\varphi$ be a d.s.h.\ function on $\mathbb{P}^k$ with $\|\varphi\|_\DSH\le1$. Then there exist two functions $\varphi_+$ and $\varphi_-$ which are p.s.h.\ on $B$ and such that
$$\|\varphi_\pm\|_\DSH\le C,\qquad\varphi_\pm\le0,\qquad\text{and}\qquad\varphi=\varphi_+-\varphi_-.$$
Here, $C$ is a positive constant independent of $\varphi$.
\end{lemma}

\begin{proof}
We recall from \cite{din-sib-cmh} that, for every d.s.h.\ function $\varphi$ on $\P^k$ with $\|\varphi\|_\DSH\le 1$, there exist two quasi-p.s.h.\ functions $\varphi^+$, $\varphi^-$ and a constant $m$ such that
\begin{equation} \label{original-split}
    \varphi =\varphi^+-\varphi^- +m,\quad\max \varphi^\pm =0,\quad\norm{\varphi^\pm}_\DSH \leq c_1,\quad\ddc\varphi^\pm \geq -c_1\omega_\FS,\quad\text{and}\quad|m|\leq c_1,
\end{equation}
where $c_1>0$ is a constant independent of $\varphi$.

We fix a smooth function $v$ on $\mathbb{P}^k$ such that $\ddc v=\omega_\FS$ on $B$, and we set $\widetilde{C}=\sup_{\mathbb{P}^k} |v|$. We have $\|v\|_\DSH\lesssim \|v\|_{\Cc^2}\lesssim 1$. Define
$$\varphi_\pm:=\varphi^\pm+c_1v-c_1\widetilde{C}-(\text{sign}(m)\mp1)m/2.$$
It follows from \eqref{original-split} that $\varphi_+$ and $\varphi_-$ satisfy the desired properties. The proof is complete.
\end{proof}

We will also need the following regularizations of bounded p.s.h.\ functions.

\begin{lemma} \label{lem-regular}
Let $g:\C^k\to\R$ be a bounded function which is p.s.h.\ on $\B_k(0,r)$ for some $r>4$. There exists a sequence of smooth functions $h^{(\ell)}$ on $\C^k$ with compact support in $\B_k(0,r+2)$, which are p.s.h.\ on $\B_k(0,r-1)$, and satisfy
$$\norm{h^{(\ell)}-g}_{L^1(\nu)}\leq 1/\ell,\quad \norm{h^{(\ell)}}_\infty \leq c\norm{g}_\infty,\quad\text{and}\quad\norm{h^{(\ell)}}_{\Cc^2(\overline{\B}_k(0,r-2)^c)}\leq c\norm{g}_\infty,$$
where $\nu$ is a given finite measure with support contained in $\B_k(0,r-4)$, and $c>0$ is a constant independent of $g$ and $\ell$.
\end{lemma}

\begin{proof}
Up to possibly adding some constant $c_0\lesssim\|g\|_\infty$ to $g$, we may assume $g\geq 0$. Define
$$\tau(z):=\norm{g}_\infty \cdot \log{\frac{\norm{z}}{r-4}} \cdot \Big( \log {\frac{r-3}{r-4}} \Big)^{-1}\qquad\text{and}\qquad\widetilde h(z):=\chi(\norm{z})  \cdot\max \big( g(z),\tau(z)   \big), $$
where $\chi$ is a real cut-off function on $\R$ satisfying
$$\chi(t)=0 \,\text{ for }\, t\geq r+1,\quad \chi(t)=1 \,\text{ for }\, t\leq r,\quad|\chi'|\leq 2,\quad\text{and}\quad|\chi''|\leq 4.$$

 When $\norm{z}\leq r$, we have $\widetilde h=\max(g,\tau)$. Hence, as both $g$ and $\tau$ are p.s.h., $\widetilde h$ is p.s.h.\ on $\B_k(0,r)$. When $\norm{z}\geq r-3$, we have
$$\log{\frac{\norm{z}}{r-4}} \cdot \Big( \log {\frac{r-3}{r-4}} \Big)^{-1}\geq 1.$$
In this case, we have $\widetilde h(z)=\chi(\norm{z})\cdot \tau(z)$, so $\widetilde h$ is smooth on $\overline{\B}_k(0,r-3)^c$. When $\norm{z}\leq r-4$, we have $\tau(z)\leq 0 \leq g(z)$, which implies that $\widetilde h=g$ on $\B_k(0,r-4)$.
Moreover, we have
$$\norm{ \widetilde h}_\infty \leq \max \big\{\norm{g}_\infty, \norm{\tau\cdot\mathds{1}_{\B_k(0,r+1)}}_\infty \big\}\lesssim \norm{g}_\infty$$
and 
$$ \norm{\widetilde h}_{\Cc^2(\overline{\B}_k(0,r-3)^c)}=\big\| \chi(\norm{z})\cdot \tau(z)\big\|_{\Cc^2(\overline{\B}_k(0,r-3)^c)}\lesssim \norm{g}_\infty.$$

Lastly, since $\widetilde h$ is bounded and p.s.h.\ on $\B_k(0,r)$, using a convolution, we can find a sequence of smooth functions $h^{(\ell)}$ with compact support in $\B_k(0,r+2)$, p.s.h.\ on $\B_k(0,r-1)$, such that
$$\norm{h^{(\ell)}}_{\Cc^2(\overline{\B}_k(0,r-2)^c)}\le 2\norm{\widetilde h}_{\Cc^2(\overline{\B}_k(0,r-3)^c)}\lesssim \norm{g}_\infty,$$
and converging pointwise to $\widetilde h$. By Lebesgue's dominated convergence theorem, they converge to $g$ in $L^1(\nu)$. The proof is complete.
\end{proof}

A positive measure $\nu$ on $\P^k$ is said to be \textit{moderate} if, for any bounded family $\Fc$ of d.s.h.\ functions on $\P^k$, there exist constants $\alpha>0$ and $c>0$ such that 
\begin{equation*}\label{mod}
\nu \big\{z\in \P^k: \,|\psi(z)|>M \big\}\leq ce^{-\alpha M}
\end{equation*} 
for every $M\geq 0$ and $\psi\in\Fc$, see \cite{dinh-exponential,dinh-dynamique,dinh-sibony:cime}. Moderate measures are PB. We have the following result.

\begin{lemma} \label{tail-split}
Let $\varphi$ be a non-positive d.s.h.\ function on $\P^k$, with $\|\varphi\|_\DSH\le1$, and p.s.h.\ on a bounded domain $B$ in the chart $\C^k\subseteq\P^k$. Let $\nu$ be a moderate measure on $\mathbb{P}^k$. For every $N\ge0$, we can write $\varphi=\varphi^{(N)}_1+\varphi^{(N)}_2$ with: $\varphi^{(N)}_1$ p.s.h.\ on $B$, $$\|\varphi^{(N)}_1\|_\infty\le N, \quad\text{and}\quad\|\varphi^{(N)}_2\|_{L^q(\nu)}\le C_qe^{-\alpha N/q}$$ for every $q\ge1$, where $\alpha>0$ is a constant independent of $\varphi$ and $q$, and $C_q>0$ is a constant independent of $\varphi$.
\end{lemma}

\begin{proof}
Up to changing the constants, we can suppose $N$ to be an integer. By the moderate property of $\nu$, we can find constants $\alpha>0$ and $c>0$ such that 
$$\nu \big\{z\in \P^k: \, -\psi(z)>M \big\}=\nu\big\{z\in \P^k: \,|\psi(z)|>M\big\}\leq ce^{-\alpha M}$$
for all $M>0$ and all non-positive $\psi$ such that $\|\psi\|_\DSH\le1$. We define $$\varphi^{(N)}_1:=\max\{\varphi,-N\},  \quad \varphi^{(N)}_2:=\varphi-\varphi^{(N)}_1,$$ 
so that $\varphi^{(N)}_1$ is p.s.h.\ on $B$ and $\|\varphi^{(N)}_1\|_{\infty}\le N$. Since $\varphi^{(N)}_2=0$ on $\{\varphi\ge-N\}$, for every $q\ge1$ we have
\begin{align*}
    \int |\varphi^{(N)}_2|^q\diff\nu&=\int_{\{\varphi<-N\}}|\varphi^{(N)}_2|^q\diff\nu \le \sum_{M\ge N} \int_{\{-M-1\le\varphi<-M\}} (M-N+1)^q\diff\nu\\
    &\le c\sum_{M\ge N} (M-N+1)^qe^{-\alpha M}=ce^{-\alpha N}\sum_{M\ge0} (M+1)^qe^{-\alpha M}=C'_qe^{-\alpha N},
\end{align*}
where $C'_q:=c\sum_{M\ge0} (M+1)^qe^{-\alpha M}$ is a constant depending only on $q$. Therefore, we have $\|\varphi^{(N)}_2\|_{L^q(\nu)}\leq C_qe^{-\alpha N/q}$, with $\alpha$ and $C_q$ as in the statement, and the proof is complete.
\end{proof}

We now introduce a norm on the space of real currents with support in $\overline U$, where $U$ is an open set in $\mathbb{P}^k$. Let $\Omega$ be a real $(p+1,p+1)$-current  supported in $\overline U$ as above and assume that there exists a positive closed $(p+1,p+1)$-current $\Omega'$ supported in $\overline U$ such that $-\Omega'\leq \Omega\leq\Omega'$. Define the norm $\|\Omega\|_{*,U}$ as
\[\|\Omega\|_{*,U}:=\min\big\{\|\Omega'\|,\Omega' \text{ as above} \big\}.\]
We have the following simple fact.

\begin{lemma} \label{lem-oneside}
	Let $U$ be an open set in $\mathbb{P}^k$ and $\Omega$ be a real $\ddc$-exact $(p+1,p+1)$-current  supported in $\overline U$. Assume that we have $\Omega\geq-R$ for some positive closed current $R$ supported in $\overline U$. Then we have $\|\Omega\|_{*,U}\leq 2\|R\|$.
\end{lemma}

\begin{proof}
	Note that $\Omega+2R$ is a positive closed current  supported in $\overline U$ and $\Omega$ satisfies \[-(\Omega+2R)\leq \Omega\leq \Omega+2R.\] The mass of $\Omega+2R$ is equal to $2\|R\|$ because $\Omega$ is $\ddc$-exact. The assertion follows from the definition of $\norm{\cdot}_{*,U}$.
\end{proof}

We will use the definition of $\norm{\cdot}_{*,U}$ and the lemma above also with $\mathbb{P}^{2k}$ instead of $\P^k$.

\subsection{Dynamical properties of H\'enon-Sibony maps}

Let now
$f$ be a polynomial automorphism of $\C^k$. It can be extended to a birational map of $\P^k$ in a natural way, still denoted by $f$. For the new map, the set $I_+$ (resp.\ $I_-$) where $f$ (resp.\ $f^{-1}$) is not defined is called the \textit{indeterminacy set} of  $f$ (resp.\ $f^{-1}$). We say that $f$ is a \textit{H\'enon-Sibony map} if $I_{+}$ and $I_-$ are non-empty and  $I_{+}\cap I_{-}=\varnothing$. Such maps were introduced in \cite{sibony:panorama}, and reduce to the Hénon maps for $k=2$. We recall here their main properties, and refer to \cite{rigidity,sibony:panorama} for details.

The indeterminacy sets $I_{\pm}$ are analytic subsets of the hyperplane at infinity $L_{\infty}:=\P^k\backslash \C^k$. For every H\'enon-Sibony map $f$, there exists a positive integer $p=p(f)$ such that $\dim I_+=k-1-p$ and $\dim I_-=p-1$. The  set $I_-$ is attracting for $f$ and $I_+$ is attracting for $f^{-1}$. Moreover, we have $f(L_\infty\backslash I_+)=I_-$ and $f^{-1}(L_\infty\backslash I_-)=I_+$. Denote by $d_+$ and $d_-$  the algebraic degrees of $f$ and $f^{-1}$ respectively. We have $d_+^p=d_-^{k-p}$. In particular, when $k=2p$, we have $d_+=d_-$.

The \textit{Green functions} of $f$ are defined by 
$$G^+(z):=\lim_{n\to\infty} d_+^{-n}\log^+\|f^n(z)\|\quad\text{and}\quad G^-(z):=\lim_{n\to\infty} d_-^{-n}\log^+\|f^{-n}(z)\|,$$
where $\log^+(\cdot):=\max\{\log(\cdot),0\}$. They are H\"older continuous  and p.s.h.\ on $\C^k$, and they satisfy $G^+\circ f=d_+G^+$ and $G^-\circ f^{-1}=d_-G^-$. Define the \textit{Green currents} of bidegree $(1,1)$ of $f$ by $T_+:=\ddc G^+$ and $T_-:=\ddc G^-$. We have $f^*(T_+)=d_+T_+$ and $f_*(T_-)=d_-T_-$ as currents  on $\P^k$.

The current $T^p_+$ (resp. $T^{k-p}_-$) is supported in the boundary of the \textit{filled Julia set} $K_+$ (resp.\ $K_-$). Recall that $K_+$ (resp.\ $K_-$) is the set of points $z\in \C^k$ whose orbit $\big(f^n(z)\big)_{n\in\N}$ (resp.\ backward orbit $\big(f^{-n}(z)\big)_{n\in\N}$) is bounded in $\C^k$. We have $$K_+=\{G^+=0\},\quad K_-=\{G^-=0\}, \quad\text{and} \quad\overline K_\pm \cap L_\infty=I_\pm.$$ The open set $\P^k\backslash\overline K_+$ (resp. $\P^k \backslash \overline K_-$) is the immediate basin of $I_-$ for $f$ (resp.\ $I_+$ for $f^{-1}$). Set $K:=K_+\cap K_-$. It is a compact subset of $\C^k$.

\medskip

Throughout this article, we fix a large $r>0$ such that $K\subset \B_k(0,r/3)$, and such that there exist open sets $U_j,V_j\subseteq \P^k$, $j=1,2$, with $U_2\cap V_2\subset\B_k(0,r/2)$,
\[\overline K_+\Subset U_1\Subset U_2, \qquad  \overline K_-\Subset V_1\Subset V_2, \qquad f^{-1}(U_i)\Subset U_i,  \qquad  f(V_i)\Subset V_i,\]
and
\begin{equation} \label{iterate-julia}
K_- \cap \B_k(0,r-1) \subset f(\B_k(0,r-1)), \quad  K_+\cap\B_k(0,r-1) \subset f^{-1}(\B_k(0,r-1)).
\end{equation}
Such sets exist since $\P^k\backslash\overline K_+$ (resp. $\P^k \backslash \overline K_-$) is the immediate basin of $I_-$ for $f$ (resp. $I_+$ for $f^{-1}$).

\medskip

We have the following crucial estimate involving the $*$-norm, see for instance \cite[Proposition 2.1]{dinh-decay-henon} and \cite[Corollary 2.5]{bian-dinh-sigma}.

\begin{lemma}  \label{lem-key}
Let $R$ be a positive closed $(k-p,k-p)$-current of mass $1$ supported in $V_1$ and $\Phi$ be a $\Cc^2$ function with compact support in $\C^k$ such that $\ddc \Phi\geq 0$ on a neighborhood of $K_+\cap V_2$. Then there exists a constant $c>0$ independent of $R$ and $\Phi$ such that 
$$\big \lp  d_-^{-(k-p)n} (f^n)_*(R) - T_-^p  \,,\, \Phi T_+^p  \big\rp \leq cd_-^{-n} \norm{\ddc \Phi \wedge T_+^p}_{*,U_1}  \quad\text{for all}\quad n\geq 0.$$
\end{lemma}

\medskip

Sibony also introduced the \textit{equilibrium measure} $\mu:=T_+^p\wedge T_-^{k-p}$ for $f$, which is an invariant probability measure. We have $\supp(\mu)\subseteq K$ and $\mu$ is the unique measure of maximal entropy of $f$ \cite{deth-auto}. As the Green currents have H\"older continuous potentials, $\mu$ is moderate.

\medskip

\section{Mixing for d.s.h.\ functions} \label{sec-mixing}

In this section, we are going to prove Theorem \ref{thm-mixing}. We first prove a mixing proposition for bounded p.s.h.\ functions, which improves the main result of \cite{hen-gab-CLT} in our context. After that, we show that the singularities of unbounded d.s.h.\ functions are ``asymptotically negligible'' in the integrals of Definition \ref{def-mix}, thus obtaining the exponential mixing of all orders also for those functions.

\medskip

\textbf{For simplicity, we will assume} $k=2p$, and set $d:=d_+=d_-$. For the case $k\neq 2p$, the proof can be adapted by working on $\P^k\times \P^k$ instead of $\P^{2k}$, see \cite{vigny-decay}. We   recall here that we take $r$ to be sufficiently large, so that \eqref{iterate-julia} holds. We also fix, in this section, a large domain $B$ containing $\overline{\mathbb{B}}_k(0,r)$.

\subsection{Mixing for bounded p.s.h.\ functions}

Our goal in this subsection is to prove the following version of Theorem \ref{thm-mixing} for bounded p.s.h.\ observables.

\begin{proposition} \label{prop-mixing-bounded}
For every $\kappa\in\mathbb{N}^*$ there exists a constant $C_\kappa>0$ such that, for every   $\kappa+1$ bounded functions $g_0,g_1, \dots,g_\kappa$ which are p.s.h.\  on $B$, and every $0=n_0\leq n_1\leq \ldots\leq n_\kappa$, we have
$$\Big|\int g_0  ( g_1\circ f^{n_1}  ) \cdots (g_\kappa \circ f^{n_\kappa}  )\,\dd\mu -\prod_{j=0}^\kappa \int  g_j  \,\dd\mu\Big|\leq C_\kappa \, d^{ -\min_{0\leq j\leq \kappa-1}(n_{j+1}-n_j) /2   }  \prod_{j=0}^\kappa \norm{g_j}_\infty.$$
\end{proposition}

Observe that, by linearity, the statement is true if we assume that for every $j$ either $g_j$ or $-g_j$ is p.s.h..

\begin{remark} \label{somewhat-better}  \rm
Comparing to the second statement in \cite[Theorem 1.1]{hen-gab-CLT}, Proposition \ref{prop-mixing-bounded} only requires a weaker infinity-norm for the $g_j$'s, and the convergence speed is the optimal one known for two observables \cite{dinh-decay-henon, Wu-Ergodic}.
\end{remark}

\smallskip

Consider the canonical inclusions of $\C^k$ and $\C^k\times\C^k$ in $\P^k$ and $\P^{2k}$, respectively. We will use $z,w$ and $(z,w)$ for the canonical coordinates of $\C^k$ and $\C^k\times \C^k$, and write $[z:t],[w:t]$ and $[z:w:t]$ for the homogeneous coordinates. Denote by $\mathbb L_\infty$  the hyperplane at infinity of $\P^{2k}$.

Define a new automorphism of  $\C^k\times \C^k$ by $F(z,w):=\big(f(z),f^{-1}(w)\big)$.
Then $F$ is also a H\'enon-Sibony map, see for instance \cite[Lemma 3.2]{dinh-decay-henon}. The algebraic degrees of $F$ and $F^{-1}$ are both equal to $d$. We define $\mathbb T_+:=T_+\otimes T_-$ and $\mathbb T_-:= T_-\otimes T_+$. The \emph{Green currents of bidegree $(2p,2p)=(k,k)$} of $F$ are $\mathbb T_+^p$ and $\mathbb T_-^p$, which satisfy
$$F^*(\mathbb T_+^p)=d^{2p}\mathbb T_+^p\quad\text{and}\quad F_*(\mathbb T_-^p)=d^{2p}\mathbb T_-^p\quad\text{as currents on}\quad\P^k\times\P^k.$$
The measure $\mu \otimes \mu= \mathbb T_+^p \wedge \mathbb T_-^p$ is invariant under $F$.
Moreover, notice that 
$$F_* (\mathbb T_+^p)=d^{-2p}\mathbb T_+^p \quad \text{as currents on}\quad\C^k\times \C^k$$ 
since $F$ is an automorphism of $\C^k\times \C^k$.
Let $\mathbb K_+$ and $\mathbb K_-$ be the filled Julia sets of $F$ and $F^{-1}$, which are equal to $K_+\times K_-$ and $K_-\times K_+$ respectively. 

Let  $\mathbb I_\pm$ be the indeterminacy sets of $F$ and $F^{-1}$. Let also $\Delta$ be the diagonal of $\C^k\times\C^k$ and let $\overline\Delta$ be its closure in $\P^{2k}$. We have $\mathbb I_\pm \cap \overline\Delta=\varnothing$.

Until the end of the section, we fix $r$ sufficiently large in order to have \eqref{iterate-julia}, and also so that there exist open sets $\mathbb U_j,\mathbb V_j\subseteq\P^{2k}$, $j=1,2$, with $\mathbb U_2\cap \mathbb V_2\subset\B_{2k}(0,r/2)$,
\[\overline{\mathbb K}_+\Subset \mathbb U_1\Subset\mathbb U_2, \qquad  \overline{ \mathbb K}_-\Subset \mathbb V_1\Subset\mathbb V_2, \qquad 
F^{-1}(\mathbb U_i)\Subset \mathbb U_i,  \qquad  F(\mathbb V_i)\Subset \mathbb V_i,\]
and
\begin{equation} \label{iterated-julia-2}
    \mathbb K_- \cap \B_{2k}(0,r-1) \subset F(\B_{2k}(0,r-1)),\quad\mathbb K_+\cap\B_{2k}(0,r-1) \subset F^{-1}(\B_{2k}(0,r-1)).
\end{equation}
Since $\mathbb I_+ \cap \overline \Delta=\varnothing$, we can also require $\mathbb V_1\supseteq\overline{\mathbb K}_-\cup\overline\Delta$.

\medskip

We set $\widetilde g_0:=g_0$ and $\widetilde{g}_j:=g_j\circ f^{n_j-n_1}$ for $j\geq 1$. Define
$$G_0(z,w):=g_0(w)\quad\text{and}\quad G_j(z,w)=g_j(z) \,\text{ for }j\ge 1.$$
Notice that $\|G_j\|_\infty=\|g_j\|_\infty$ for every $j$. Since the $g_j$'s are p.s.h.\ on $\B_k(0,r)$, we have that $G_0$ is p.s.h.\ on $\C^k\times\B_k(0,r) \supseteq \B_{2k}(0,r)$, and the $G_j$'s are p.s.h.\ on $\B_k(0,r)\times\C^k\supseteq \B_{2k}(0,r)$ for $j\ge1$.

Recall that $\supp(\mu)$ is contained in $K$, which is invariant under $f$. This gives $\supp(\mu)\times\supp(\mu)\subseteq K\times K=\mathbb K_+\cap\mathbb K_-$, which is invariant under $F$. Since we will always integrate the $G_j$'s against finite measures with support contained in $\supp(\mu)\times\supp(\mu)$, we can apply Lemma \ref{lem-regular} (with $2k$ instead of $k$) to the $G_j$'s with $\ell_0$ sufficiently large so that
$$1/\ell_0 \ll \big(d\cdot(\norm{F}_{\Cc^1(K\times K)}+1)\big)^{ -n_\kappa}$$
and obtain a regularization $H^{(\ell_0)}_j$ for every $G_j$. Therefore, up to replacing the $G_j$'s with the $H^{(\ell_0)}_j$'s, from now on we can assume that the functions $G_j$'s are \textbf{smooth, with compact support in} $\B_{2k}(0,r+2)$, \textbf{p.s.h.} on $\B_{2k}(0,r-1)$, and satisfy
\begin{equation} \label{G-estimates}
|G_j|\leq 1 \qquad\text{and}\qquad\norm{G_j}_{\Cc^2(\overline{\B}_{2k}(0,r-2)^c)} \leq 1.
\end{equation}

Set $\widetilde G_0:=G_0$ and $\widetilde G_j:=G_j\circ F^{n_j-n_1}$ for $j\geq 1$. Define the auxiliary functions $\Phi^\pm$ on $\C^k\times \C^k$ by:
\begin{equation}\label{phipm}
\Phi^\pm:=\Phi^\pm_{n_0,\dots,n_\kappa}=\sum_{j=0}^\kappa \Big((\kappa+1)\widetilde{G}_j+\frac{\kappa}{2}\widetilde{G}^2_j\Big)\pm \prod_{j=0}^\kappa \widetilde{G}_j.
\end{equation}
Notice that these two functions have compact support in $\C^k\times\C^k$.
We have the two following lemmas about the functions $\Phi^\pm$. The delicate point of these estimates is the independence of the $n_j$'s in Lemma \ref{lem-star-bound}.

\begin{lemma} \label{ddc-positive}
   We have $\ddc \Phi^\pm \geq 0$ on a neighborhood of $\mathbb K_+ \cap \B_{2k}(0,r-1)$, possibly depending on $n_1,\dots,n_\kappa$, but not on $g_0,\dots,g_\kappa$.
\end{lemma}

\begin{proof}
    Remember that the inequality $i\del(g\pm h)\wedge\dbar(g\pm h)\ge0$, which is valid for every $\Cc^2$ functions $g$ and $h$, implies
    \begin{equation} \label{rearrangement}
    \pm(i\del g\wedge\dbar h+i\del h\wedge\dbar g)\ge-(i\del g\wedge\dbar g+i\del h\wedge\dbar h).
    \end{equation}
    It follows that we have
    \begin{align}
        i\ddbar\Phi^\pm=&\sum_{j=0}^\kappa i\ddbar\widetilde{G}_j\Big(\kappa+1+\kappa\widetilde{G}_j\pm\prod_{s \not=j}\widetilde{G}_s\Big)+\kappa\sum_{j=0}^\kappa i\del\widetilde{G}_j\wedge\dbar\widetilde{G}_j
        \pm \sum_{j\not=s}\Big(i\del \widetilde{G}_j\wedge\dbar \widetilde{G}_s\prod_{t\neq j,s}\widetilde{G}_t\Big)  \nonumber\\
        & \ge \sum_{j=0}^\kappa i\ddbar\widetilde{G}_j\Big(\kappa+1+\kappa\widetilde{G}_j\pm\prod_{s\not=j}\widetilde{G}_s\Big)+\sum_{j=0}^\kappa i\del \widetilde{G}_j\wedge\dbar \widetilde{G}_j\bigg(\kappa-\sum_{s\not=j}  \Big(\prod_{t\neq j,s}|\widetilde{G}_t| \Big)\bigg)\nonumber\\
        &\ge \sum_{j=0}^\kappa i\ddbar\widetilde{G}_j\Big(\kappa+1+\kappa\widetilde{G}_j\pm\prod_{s\not=j}\widetilde{G}_s\Big), \label{ddcphipm}
    \end{align}
    where in the first inequality we have used \eqref{rearrangement}, and in the second one we have used the estimates \eqref{G-estimates}.
    
    It follows from \eqref{iterated-julia-2} that $\mathbb K_+\cap\B_{2k}(0,r-1) \subset F^{-n}(\B_{2k}(0,r-1))$ for every $n\in\mathbb N$. So, for every $1\le j\le\kappa$ the function $\widetilde{G}_j$ is p.s.h.\ on a neighborhood $U$ (depending on the $n_j$'s) of $\mathbb K_+\cap \B_{2k}(0,r-1)$. Therefore, we have
    \begin{equation} \label{lower-bound-Phi}
        i\ddbar\Phi^\pm \geq \sum_{j=0}^\kappa i\ddbar\widetilde{G}_j\Big(\kappa+1+\kappa\widetilde{G}_j\pm\prod_{s\not=j}\widetilde{G}_s\Big) \ge 0\quad\text{on}\quad U,
    \end{equation}
    where the first inequality is given by \eqref{ddcphipm}, and the last one follows from the estimates \eqref{G-estimates}. The assertion follows.
\end{proof}

\begin{lemma} \label{lem-star-bound}
We have $\norm{\ddc\Phi^\pm\wedge\mathbb T_+^p}_{*,\mathbb U_1} \leq c \kappa^2$ for some $c>0$ independent of $n_1,\dots, n_\kappa$ and $g_0,\dots,g_\kappa$.
\end{lemma}

\begin{proof}
From  \eqref{ddcphipm}, we see that we have
$$\ddc \Phi^\pm \geq \sum_{j=0}^\kappa \tau_j \ddc \widetilde G_j,$$
for some functions $\tau_j$'s on $\C^k \times \C^k$ satisfying $0\leq \tau_j\leq 2\kappa+2$.
For every $0\leq j\leq \kappa$, the assumptions \eqref{G-estimates} on the $G_j$'s give that $\ddc G_j\geq -c \omega_\FS$ on $\P^{2k}$ for some constant $c>0$ independent of $j$. It follows that
$$\ddc \widetilde G_j \geq -c' (F^{n_j-n_1})^* \omega_\FS$$
for some constant $c'>0$ independent of $n_1, \dots, n_\kappa$ and $g_0,\dots,g_\kappa$.
 Thus, we deduce that 
\begin{equation} \label{*estimate1}
\ddc\Phi^\pm \wedge \mathbb T_+^p \gtrsim - \kappa\Big(\omega_\FS+\sum_{j=1}^\kappa(F^{n_j-n_1})^*\omega_\FS \Big)\wedge \mathbb T_+^p,
\end{equation}
 We will show that the mass of $(F^{n_j-n_1})^*\omega_\FS \wedge \mathbb T_+^p$ is bounded independently of $n_1,\dots, n_\kappa$. Using that $F^*(\mathbb T_+^p)=d^{2p}\mathbb T_+^p$, we have 
\begin{equation} \label{*estimate2}
(F^{n_j-n_1})^*\omega_\FS \wedge \mathbb T_+^p  = d^{-2p(n_j-n_1)}(F^{n_j-n_1})^*( \omega_\FS \wedge \mathbb T_+^p).
\end{equation}

By definition of mass, we have
\begin{align*}
\big\| (F^{n_j-n_1})^*( \omega_\FS \wedge \mathbb T_+^p) \big\|
&=\big\lp (F^{n_j-n_1})^*( \omega_\FS \wedge \mathbb T_+^p)  \,,\, \omega_\FS^{2p-1} \big\rp \\
&= \big\lp   \omega_\FS \wedge \mathbb T_+^p  \,,\,  (F^{n_j-n_1})_*  \omega_\FS^{2p-1} \big\rp.
\end{align*}
As the mass of a positive closed current on $\P^{2k}$ only depends on its cohomology class, we conclude that
\begin{equation}\label{*estimate3}
\big\| (F^{n_j-n_1})^*( \omega_\FS \wedge \mathbb T_+^p) \big\|=\big\| (F^{n_j-n_1})_*  \omega_\FS^{2p-1}  \big\| =d^{(2p-1)(n_j-n_1)}
\end{equation}
since both $\omega_\FS$ and $\mathbb T_+$ have mass $1$. We finish the proof by using \eqref{*estimate1} to apply Lemma \ref{lem-oneside} with the exact current $\ddc\Phi^\pm\wedge\mathbb T^p_+$ instead of $\Omega$, and combining \eqref{*estimate2} and \eqref{*estimate3}.
\end{proof}

We now begin the proof of Proposition \ref{prop-mixing-bounded}. Using the invariance of $\mu$, the desired inequality does not change if we replace $n_j$ by $n_j-1$ for $1\leq j\leq \kappa$ and $g_0$ by $g_0\circ f^{-1}$. Therefore, it is enough to  assume that $n_1$ is even. We have the following lemma.

\begin{lemma}\label{lem-Psi}
There is a constant $c_\kappa>0$, independent of $n_1,\dots, n_\kappa$ and $g_0,\dots,g_\kappa$, such that 
$$ \Big| \int \prod_{j=0}^\kappa (g_j\circ f^{n_j})\,\dd \mu -\int g_0 \,\dd \mu \int \prod_{j=1}^\kappa (g_j\circ f^{n_j-n_1})\,\dd \mu \Big|\leq c_\kappa d^{-n_1 /2}.    $$
\end{lemma}

\begin{proof}
Put $\Psi:= g_1 (g_2 \circ f^{n_2-n_1})\cdots (g_\kappa \circ f^{n_\kappa-n_1})=\prod_{j=1}^\kappa \widetilde{g}_j$. 
We are going to prove that we have
\begin{equation} \label{ineq-plus}
 \int g_0 ( \Psi\circ f^{n_1})\,\dd \mu -\int g_0 \,\dd \mu \int \Psi \,\dd \mu \leq c_\kappa d^{-n_1 /2}  \end{equation}
and 
\begin{equation} \label{ineq-minus} 
-\int g_0 ( \Psi\circ f^{n_1})\,\dd \mu +\int g_0 \,\dd \mu \int \Psi \,\dd \mu \leq c_\kappa d^{-n_1 /2},   \end{equation}
for some $c_\kappa>0$ independent of $n_1,\dots, n_\kappa$ and $g_0,\dots,g_\kappa$. This gives the desired result. We will make use of the functions $\Phi^\pm$ defined in \eqref{phipm}.

\smallskip

Recall that $\mu=T_+^p \wedge T_-^p$, and that the currents $T_\pm$ have H\"older continuous potentials on $\C^k$. It follows that the intersections of $T_\pm$ with positive closed currents on $\C^k$ are well defined.
Using the invariance of $\mu$, we have
\begin{align*}
A:&=\int g_0 (\Psi \circ f^{n_1})\,\dd\mu+\int \Big( (\kappa+1)\sum_{j=0}^\kappa g_j+\frac{\kappa}{2}\sum_{j=0}^\kappa g^2_j \Big)\,\dd\mu\\
&=\begin{aligned}[t]\bigg\lp& T_+^p\wedge T_-^p \,,\, (\kappa+1)g_0\circ f^{-n_1/2}+\frac{\kappa}{2}g_0^2\circ f^{-n_1/2}\\
&+\sum_{j=1}^\kappa \Big((\kappa+1)\widetilde{g}_j\circ f^{n_1/2}+\frac{k}{2}\widetilde{g}_j^2\circ f^{n_1/2}\Big)+(g_0\circ f^{-n_1/2})(\Psi\circ f^{n_1/2})\bigg\rp\end{aligned}\\
&=\begin{aligned}[t]\bigg\lp &(T_+^p \otimes T_-^p) \wedge [\Delta]  \,,\, (\kappa+1)g_0\big(f^{-n_1/2}(w)\big)+\frac{\kappa}{2}g_0^2\big(f^{-n_1/2}(w)\big)\\
&+\sum_{j=1}^\kappa \Big((\kappa+1)\widetilde{g}_j\big(f^{n_1/2}(z)\big)+\frac{k}{2}\widetilde{g}_j^2\big(f^{n_1/2}(z)\big)\Big)+g_0\big(f^{-n_1/2}(w)\big)\Psi\big(f^{n_1/2}(z)\big)\bigg\rp.\end{aligned}
\end{align*}

From the definitions of $\Psi$ and $\Phi^+$ and the fact that $F_*(\mathbb T_+^p)=d^{-2p}\mathbb T_+^p$  as currents on $\C^k \times \C^k$, it follows that
\begin{align*}
A&= \big\lp \mathbb T_+^p \wedge [\Delta]  \,,\, \Phi^+\circ F^{n_1/2} \big\rp  = \big\lp \mathbb T_+^p \wedge [\Delta]  \,,\, (F^{n_1/2})^* \Phi^+ \big\rp  \\
&=  \big\lp d^{-pn_1} \mathbb T_+^p \wedge (F^{n_1/2})_*[\Delta]  \, ,\, \Phi^+ \big\rp = \big\lp  d^{-pn_1}(F^{n_1/2})_*[\Delta]  \,,\,   \Phi^+  \mathbb T_+^p \big\rp.
\end{align*}

Therefore we have
\begin{equation} \label{split1}
    \int g_0 (\Psi \circ f^{n_1})\,\dd\mu+\int \Big( (\kappa+1)\sum_{j=0}^\kappa g_j+\frac{\kappa}{2}\sum_{j=0}^\kappa g^2_j \Big)\,\dd\mu = \big\lp  d^{-pn_1}(F^{n_1/2})_*[\Delta]  \,,\,   \Phi^+  \mathbb T_+^p \big\rp.
\end{equation}

\smallskip

Recall that   $\overline{\mathbb K}_-\cup\overline\Delta\subseteq \mathbb V_1\Subset\mathbb V_2$, and $F(\mathbb V_j)\Subset \mathbb V_j$ for $j=1,2$. Lemma  \ref{ddc-positive} implies that we have $\ddc \Phi^+\geq 0$ on a neighborhood of $\mathbb K_+\cap \mathbb V_2$. Thus, applying Lemma \ref{lem-key} with $F,[\Delta], \Phi^+, \mathbb K_+,\mathbb V_j, \mathbb T_\pm$ instead of $f,R, \Phi, K_+,V_j, T_\pm$, yields
\begin{equation} \label{split3}
\big\lp  d ^{-pn_1} (F^{n_1/2})_*[\Delta] - \mathbb T_-^p \,,\, \Phi^+(z,w)\mathbb T_+^p \big\rp \leq cd^{-n_1/2} \norm{\ddc \Phi^+ \wedge \mathbb T_+^p}_{*,\mathbb U_1}.
\end{equation}

Since $\mu\otimes \mu= \mathbb T_+^p \wedge \mathbb T_-^p$, and using also the invariance of $\mu$, we get
\begin{equation} \label{split4}
\big\lp  \mathbb T_-^p, \Phi^+ (z,w)\mathbb T_+^p\big\rp=   \lp   \mu\otimes \mu , \Phi^+ \rp =\int \Big( (\kappa+1)\sum_{j=0}^\kappa g_j+\frac{\kappa}{2}\sum_{j=0}^\kappa g^2_j \Big)\,\dd\mu+\lp \mu, g_0 \rp\lp \mu, \Psi  \rp.
\end{equation}

\smallskip

Combining \eqref{split1}, \eqref{split3}, and \eqref{split4}, we deduce that 
$$\big\lp \mu, g_0 (\Psi \circ f^{n_1}) \big\rp  -\lp \mu, g_0 \rp\lp \mu, \Psi  \rp \leq cd^{-n_1/2} \norm{\ddc \Phi^+ \wedge \mathbb T_+^p}_{*,\mathbb U_1}.$$
Thus,  \eqref{ineq-plus} follows from Lemma \ref{lem-star-bound}. Repeating the above argument with $\Phi^-$ instead of $\Phi^+$, we also get the inequality \eqref{ineq-minus}. This concludes the proof of the lemma.
\end{proof}

\smallskip

\begin{proof}[End of the proof of Proposition \ref{prop-mixing-bounded}]
We proceed by induction. The base case $\kappa=1$ is given by \cite[Theorem 1.1]{Wu-Ergodic}.  Suppose that the proposition holds for $\kappa-1$ observables. We need to prove that the statement holds for $\kappa$, i.e., that we have
$$\Big|\int \prod_{j=0}^\kappa (g_j\circ f^{n_j})\,\dd\mu -\prod_{j=0}^\kappa \int  g_j  \,\dd\mu\Big|\lesssim \, d^{ -\min_{0\leq j\leq \kappa-1}(n_{j+1}-n_j) /2   }.   $$
Recall that we can assume that $|g_j|\leq 1$ for every $j\ge1$. By Lemma \ref{lem-Psi}, it is enough to show that we have
$$ \Big| \int g_0 \,\dd \mu \int \prod_{j=1}^\kappa (g_j\circ f^{n_j-n_1})\,\dd \mu-\prod_{j=0}^\kappa \int  g_j  \,\dd\mu\Big|\lesssim \, d^{ -\min_{1\leq j\leq \kappa-1}(n_{j+1}-n_j) /2   }.$$
This follows from the inductive assumption. The proof is complete.
\end{proof}

\subsection{Mixing for all d.s.h.\ functions} \label{mixing-dsh}
We are ready to prove our first main theorem.

\begin{proof}[Proof of Theorem \ref{thm-mixing}]
Up to rescaling, we can assume without loss of generality that $\|\varphi_j\|_\DSH\le1$ for every $j$. Applying Lemma \ref{dsh-split-psh}, and by linearity, we may also assume that we have $\varphi_j\le0$ and $\|\varphi_j\|_\DSH \le 1$ for every $j$, and that all the $\varphi_j$'s are p.s.h.\ on $B$.

\smallskip

Using Lemma \ref{tail-split}, we can write $\varphi_j=\varphi^{(N)}_{j,1}+\varphi^{(N)}_{j,2}$, where we choose $N$ as
\begin{equation} \label{choiceofN}
    N:=(2\alpha)^{-1}\min_{0\leq j\leq \kappa-1}(n_{j+1}-n_j)\log{d}.
\end{equation}
Since $N$ is fixed,  we will omit its dependence and write $\varphi_{j,1}^{(N)}=\varphi_{j,1}$ and $\varphi_{j,2}^{(N)}=\varphi_{j,2}$.

\smallskip

Indexing all the possible choices of the $v_j$'s indexes in the $\varphi_{j,v_j}$'s with $\bv:=(v_0,v_1,\dots,v_\kappa)\in\{1,2\}^{\kappa +1}$, we have
\begin{align*}
    \Big|\int \Big(\prod_{j=0}^\kappa &\varphi_j\circ f^{n_j} \Big)\,\dd \mu -\prod_{j=0}^\kappa  \int \varphi_j \,\dd \mu \Big|\\
    &=\bigg|\sum_{\mathbf v\in\{1,2\}^{\kappa+1}}\bigg(\int \Big(\prod_{j=0}^\kappa \varphi_{j,v_j}\circ f^{n_j} \Big) \,\dd \mu -\prod_{j=0}^\kappa  \int \varphi_{j,v_j} \,\dd \mu \bigg)\bigg|\\
    &\leq\sum_{\mathbf v\in\{1,2\}^{\kappa+1}}\Big|\int \Big(\prod_{j=0}^\kappa \varphi_{j,v_j}\circ f^{n_j} \Big)\,\dd \mu -\prod_{j=0}^\kappa  \int \varphi_{j,v_j} \,\dd \mu\Big|\\
    &\le\begin{aligned}[t]&\Big|\int \Big(\prod_{j=0}^\kappa \varphi_{j,1}\circ f^{n_j} \Big)\,\dd \mu -\prod_{j=0}^\kappa  \int \varphi_{j,1} \,\dd \mu\Big|\\
    & \quad+\sum_{ \bv\not=(1,\dots,1)}\bigg(\Big|\int \Big(\prod_{j=0}^\kappa \varphi_{j,v_j}\circ f^{n_j} \Big)\,\dd \mu\Big|+\Big|\prod_{j=0}^\kappa  \int \varphi_{j,v_j} \,\dd \mu\Big|\bigg).\end{aligned}
\end{align*}
To estimate the last term, we treat two cases separately.

\smallskip
\textbf{Case} $\bv=(1,\dots,1)$. Since all the $\varphi_{j,1}$'s are p.s.h.\ on $B$ and $\|\varphi_{j,1}\|_\infty \le N$ for every $j$, we can apply Proposition \ref{prop-mixing-bounded} to get
\begin{align*}
\Big|\int &\varphi_{0,1}  ( \varphi_{1,1}\circ f^{n_1}  ) \cdots (\varphi_{\kappa,1} \circ f^{n_\kappa}  )\,\dd\mu -\prod_{j=0}^\kappa \int  \varphi_{j,1}  \,\dd\mu\Big|\\
&\leq  C_\kappa \, d^{ -\min_{0\leq j\leq \kappa-1}(n_{j+1}-n_j)/2}\prod_{j=0}^\kappa \|\varphi_{j,1}\|_\infty\leq C_\kappa \, d^{ -\min_{0\leq j\leq \kappa-1}(n_{j+1}-n_j)/2} N^{\kappa+1}.
\end{align*}

\textbf{Case} $\bv\not=(1,\dots,1)$. Using the estimates on the $\varphi_{j,1}$'s and $\varphi_{j,2}$'s given by Lemma \ref{tail-split}, and by Hölder's inequality, each of these terms is bounded by $N^{\kappa}e^{-\alpha N}$, up to a multiplicative constant depending only on $\kappa$.

\smallskip

Combining the estimates above, and inserting the value of $N$ as in \eqref{choiceofN}, we obtain a constant $C_\kappa>0$ independent of $n_1,\dots,n_\kappa$ such that 
\begin{align*}
\Big| \int \varphi_0 (\varphi_1\circ f^{n_1})\cdots (  \varphi_\kappa \circ f^{n_\kappa}) \,\dd \mu -&\prod_{j=0}^\kappa  \int \varphi_j \,\dd \mu \Big| \\
&\leq C_\kappa \min_{0\leq j\leq \kappa-1}(n_{j+1}-n_j)^{\kappa+1} d^{ -\min_{0\leq j\leq \kappa-1}(n_{j+1}-n_j)/2 }. 
\end{align*}
Thus, the estimate in Theorem \ref{thm-mixing} holds for any $\theta>1/\sqrt d$ (up to possibly increasing the value of $C_\kappa$), as desired.
\end{proof}

\medskip
\section{Central limit theorem for d.s.h.\ observables} \label{sec-clt}

This section is devoted to prove Theorem \ref{thm-clt}. In \cite{bian-dinh-sigma}, the authors apply \cite[Theorem 1.1]{bjo-gor-clt}, with all the norms $N_s$'s appearing there equal to $\norm{\cdot}_{\Cc^\alpha}, 0<\alpha\le2$. As the same result cannot be applied in our case, we will instead prove the following variation of \cite[Theorem 1.1]{bjo-gor-clt}, from which we will deduce our desired result.

\begin{theorem} \label{general-clt}
Fix a measure space $(X,\mu)$ and take a map $f:X\rightarrow X$ such that $\mu$ is $f$-invariant. Let $(\cA, \norm{\cdot}_\cA)$ be a space of test functions on $X$. Assume that $f^*$ is bounded with respect to $\norm{\cdot}_\cA$. Assume that we also have:
\begin{enumerate}[label=(\roman*)]
    \item \textbf{mixing}: the measure $\mu$ is exponentially mixing of all orders for observables in $\cA$;
    \item \textbf{controlled tail}: there exists a set of functions $\cB$ such that $(X,\mu,f,\cA,\cB)$ satisfies the \emph{controlled tail property}.
\end{enumerate}
Then all observables in $\cA$ satisfy the central limit theorem (in the sense of \eqref{law-conv}) with respect to $\mu$.
\end{theorem}

Let us now explain the notion (ii) in the above theorem.
Let $X,\mu,f,\cA$ be as in Theorem \ref{general-clt}. From now on, every dependence on $X$ or $\mu$ will be implicit. Let also $\cB$ be a cone inside a normed vector space of real functions (with abuse of notation, we refer to the norm as $\norm{\cdot}_\cB$), such that:
\begin{enumerate}[label=(B\arabic*)]
    \item for every $\varphi\in\oB$, we have $\|\varphi\|_\infty\le\|\varphi\|_\oB$;
    \item the exponential mixing of all orders, as in Definition \ref{def-mix}, holds for all observables in $\cB$ with $\norm{\cdot}_\cB$ instead of $\norm{\cdot}_\oE$;
    \item for every integer $t\ge2$ there exist functions $P_{t,1},\dots,P_{t,m_t}:\mathbb R^t\rightarrow \mathbb R$ such that:
    \begin{itemize}
    \item for all $x_1,\dots,x_t\in\mathbb{R}$ we have
    $$\sum_{l=1}^{m_t} P_{t,l}(x_1,\dots,x_t)=x_1 x_2\cdots x_t;$$
    \item for every $g_1,\dots,g_t\in\cB$ with $\|g_j\|_\cB\le1$, we have $$P_{t,l}(g_1,\dots,g_t)\in\cB\quad\text{and}\quad\|P_{t,l}(g_1,\dots,g_t)\|_\cB\le\tilde{c}_t,$$
    $$\text{or}\quad-P_{t,l}(g_1,\dots,g_t)\in\cB\quad\text{and}\quad\|-P_{t,l}(g_1,\dots,g_t)\|_\cB\le\tilde{c}_t.$$
    \end{itemize}
    Here, $\tilde{c}_t$, $m_t$ and the $P_{t,l}$'s depend only on $t$.
\end{enumerate}
Without loss of generality, we will assume that the $\tilde{c}_t$'s and the $m_t$'s in (B3) are non-decreasing in $t$.

\smallskip

\begin{definition} \label{ctp} \rm
We say that $(X,\mu,f,\cA,\cB)$ as above satisfies the \emph{controlled tail property} if, for every $\varphi\in\cA$ and every $M\ge0$, we can write $\varphi=\varphi_1^+-\varphi_1^-+\varphi_2$ with:
\begin{enumerate}
    \item $\varphi_1^\pm\in\cB$ and $\|\varphi_1^\pm\|_\cB\le cM\|\varphi\|_\cA$ for some $c\ge 1$ independent of $\varphi$;
    \item $\|\varphi_2\|_{L^q(\mu)}\le c_q e^{-\alpha_qM}\|\varphi\|_\cA$ for every $q\in\mathbb{N}^*$, where $\alpha_q>0$ and $c_q>0$ depend only on $q$.
\end{enumerate}
\end{definition}
We will assume without loss of generality that the $\alpha_q$'s and the $c_q$'s are non-increasing and non-decreasing in $q$, respectively.

\medskip

We first use Theorem \ref{general-clt} to complete the proof of Theorem \ref{thm-clt}.

\begin{proof}[Proof of Theorem \ref{thm-clt}]
We just need to prove that $\oA:=\DSH(\mathbb{P}^k)$, equipped with the norm $\norm{\cdot}_\oA=\norm{\cdot}_\text{DSH}$, satisfies the hypotheses of Theorem \ref{general-clt}. It is not difficult to check that $f^*$ is bounded with respect to the norm $\norm{\cdot}_\DSH$.

\smallskip

Condition (i) is given by Theorem \ref{thm-mixing}.
We now verify condition (ii). To do that, we need a set of functions $\oB$ and $\norm{\cdot}_\oB$ satisfying conditions (B1-B3), and such that $(X,\mu,f,\cA,\cB)$ satisfies the controlled tail property. 
    
We take as $\oB$ the set of bounded real functions on $\P^k$ which are p.s.h.\ on $B$, with $\norm{\cdot}_\oB:=\norm{\cdot}_\infty$. Condition (B1) is immediate. Proposition \ref{prop-mixing-bounded} gives condition (B2). Condition (B3) can be obtained with $m_t=2$ for every $t$ by using the following two functions:
$$P_{t,1}(x_1,\dots,x_t):=\frac{t}{2} \sum_{j=1}^t  x_j+\frac{t-1}{4}\sum_{j=1}^t x_j^2+\frac{1}{2}\prod_{j=1}^t x_j;$$
$$P_{t,2}(x_1,\dots,x_t):=-\frac{t}{2}  \sum_{j=1}^t  x_j-\frac{t-1}{4}\sum_{j=1}^t x_j^2+\frac{1}{2}\prod_{j=1}^t x_j.$$
Observe that $2P_{\kappa+1,1}(\widetilde{G}_0,\dots,\widetilde{G}_\kappa)=\Phi^+$ and $-2P_{\kappa+1,1}(\widetilde{G}_0,\dots,\widetilde{G}_\kappa)=\Phi^-$, where $\Phi^\pm$ are defined in \eqref{phipm}. Hence, condition (B3) follows from the same arguments as in the proof of Lemma \ref{ddc-positive}.

It remains to verify that every d.s.h.\ function $\varphi$ can be written as $\varphi=\varphi_1^+-\varphi_1^-+\varphi_2$, with $\varphi_1^+,\varphi_1^-$ and $\varphi_2$ satisfying conditions (1) and (2) in Definition \ref{ctp}. This follows from Lemmas \ref{dsh-split-psh} and \ref{tail-split}. The proof is complete.
\end{proof}

\smallskip

It remains to prove  Theorem \ref{general-clt}. Let us first recall the assumptions in \cite[Theorem 1.1]{bjo-gor-clt} (which are there stated as (2.3) to (2.6)), where the $N_s$'s are (semi-)norms on $\cA$:
\begin{itemize}
    \item \textbf{monotonicity}: $N_s(\varphi)\lesssim_s N_{s+1}(\varphi)$ for every $s\ge1$ and every observable $\varphi$;
    \item \textbf{Sobolev embedding}: $\|\varphi\|_{\infty}\lesssim_s N_s(\varphi)$ for every $s\ge1$ and every observable $\varphi$;
    \item \textbf{$\Z$-boundedness}: for every $s\ge 1$, there exists $\sigma_s>0$ such that $N_s(\varphi\circ f^n)\lesssim_s e^{\sigma_s|n|}N_s(\varphi)$ for every $n\in\mathbb{Z}$ and every observable $\varphi$;
    \item \textbf{almost multiplicative}: $N_s(\varphi_1\varphi_2)\lesssim_s N_{s+1}(\varphi_1)N_{s+1}(\varphi_2)$ for every $s\ge1$ and every observables $\varphi_1,\varphi_2$.
\end{itemize}

The monotonicity is immediate to check in our case, since we have $N_s(\varphi)=\|\varphi\|_\cA$ for every $s\in\mathbb{N}$. The $\Z$-boundedness is immediate too, because we just need to take $\sigma_s:=\log{C_f}$ for every $s\in\mathbb{N}$, where $C_f:=\max\{1,\|f^*\|_\cA\}$. We do not have the Sobolev embedding condition, as we do not only use bounded observables. We also do not have the almost multiplicative condition, neither in the general setting of observables in $\cA$, nor in the specific case of d.s.h.\ observables (recall that a product of d.s.h.\ functions is not d.s.h.\ in general).

 We note that, in \cite{bjo-gor-clt}, the Sobolev embedding and almost multiplicative conditions are used only once. The first one is used in \cite[Section 9.1, p.\ 477]{bjo-gor-clt}, for an estimate that we will prove using the controlled tail property, see Lemma \ref{sobolev} below. The second one is used in the proof of \cite[Lemma 9.1]{bjo-gor-clt}. We will show a version of that lemma for a norm satisfying the hypotheses of Theorem \ref{general-clt}, see Lemma \ref{alternate-lemma} below. The lemmas that we are going to prove are stated in a simpler setting than \cite{bjo-gor-clt}, because in our case the group acting on the space of test functions is just $\Z$ and the action is given by the iterates of $f$.

\medskip

Set $[t]:=\{1,2,\dots,t\}$. We say that $\Qc$ is a \textit{partition} of $[t]$ if the elements of $\Qc$ are disjoint subsets of $[t]$ and their union is $[t]$. Given a partition $\Qc$ of $[t]$ and $n_1,\dots,n_t\in\Z$, we denote 
$$\max \Qc:=\max \Qc(n_1,\dots,n_t)  =\max_{J\in \Qc} \big\{ |n_{j_1}-n_{j_2}|:\, j_1,j_2\in J  \big\}$$
and
$$\min \Qc:=\min \Qc(n_1,\dots,n_t)=\min_{J,L\in \Qc,J\neq L} \big\{ |n_j-n_l|:\, j\in J , l\in L \big\}.$$
These quantities are the maximum of the ``diameters'' of the sets in $\Qc$, and the minimum ``distance'' between two of these sets, respectively.

\medskip

First, we prove the estimate that in \cite[Section 9.1]{bjo-gor-clt} was proved using the Sobolev embedding condition. For a finite set $I$, the notation $|I|$ means the number of elements in $I$.

\begin{lemma} \label{sobolev}
Let $t\ge 2$ be an integer. Assume $(X,\mu,f,\cA,\cB)$ satisfies the controlled tail property as in Definition \ref{ctp}.  Then, for every partition $\Qc$ of $[t]$, every $n_1,\dots,n_t\in\Z$, every $I\subseteq [t]$, and every $\varphi\in\cA$, we have
    $$\bigg|\prod_{J\in\Qc}\Big(\int \prod_{j\in I\cap J}\varphi\circ f^{n_j}\diff\mu\Big)\bigg|\lesssim_t \|\varphi\|_\cA^{|I|}.$$
\end{lemma}
\begin{proof}
    We have
    \begin{equation} \label{simple_ineq}
    \bigg|\prod_{J\in\Qc}\Big(\int \prod_{j\in I\cap J}\varphi\circ f^{n_j}\diff\mu\Big)\bigg| \le  \prod_{J\in\Qc}\bigg(\int \Big|\prod_{j\in I\cap J}\varphi\circ f^{n_j}\Big|\diff\mu\bigg).
    \end{equation}
    The controlled tail property with $M=0$ gives
    $$\|\varphi\|_{L^q(\mu)}\le c_q \|\varphi\|_\cA\quad\text{for every}\quad\varphi\in\cA, q\ge1,$$
    where $c_q>0$   depends only on $q$. Using H\"older's inequality and the invariance of $\mu$, we obtain that for every $J\in\Qc$ we have
    \begin{equation} \label{randomestimate}
        \int \Big|\prod_{j\in I\cap J}\varphi\circ f^{n_j}\Big|\diff\mu\lesssim_t \|\varphi\|_\cA^{|I\cap J|}.
    \end{equation}
    Inserting \eqref{randomestimate} in \eqref{simple_ineq} gives the statement.
\end{proof}

Finally, the following is our version of \cite[Lemma 9.1]{bjo-gor-clt}. As mentioned above, once this lemma is proved, the proof of Theorem \ref{general-clt} can be obtained by the same argument as in \cite{bjo-gor-clt}. Recall that $C_f=\max\{1,\|f^*\|_\cA\}$.

\begin{lemma} \label{alternate-lemma}
Fix $0\le a<b$, and let $t\ge 2$ be an integer.
     Assume $(X,\mu,f,\cA,\cB)$ satisfies the controlled tail property as in Definition \ref{ctp}.  Then, for every partition $\Qc$ of $[t]$, every $n_1,\dots,n_t\in\Z$ satisfying $\max\Qc\le a$ and $\min\Qc>b$, every $I\subseteq [t]$, and every $\varphi\in\cA$, we have
    \begin{equation} \label{mixedmix}
    \bigg|\int\Big(\prod_{j\in I}\varphi\circ f^{n_j}\Big)\diff\mu-\prod_{J\in\Qc}\Big(\int\Big(\prod_{j\in I\cap J}\varphi\circ f^{n_j}\Big)\diff\mu\Big)\bigg|\le C_t \eta_t^b \, C_f^{ta}\|\varphi\|_\cA^{|I|},
    \end{equation}
    where the constants $C_t>0$ and $0<\eta_t<1$  are both independent of $\Qc,a,b,\varphi$, and $I$.
\end{lemma}

    Observe that the condition $\max\Qc\le a<b<\min\Qc$ implies that the sets in $\Qc$ are ``monotone'', i.e., given any two of them, all the elements of one are greater than all the elements of the other.

\begin{proof}[Proof of Lemma \ref{alternate-lemma}]
Fix $I$ as in the statement. For every $J\in\Qc$ such that $I\cap J=\emptyset$, we put $\prod_{j\in I\cap J}\varphi\circ f^{n_j}:=1$. So we can assume, without loss of generality, that $I\cap J\not=\emptyset$ for every $J\in\Qc$.

Since $\mu$ is $f$-invariant,  we can rewrite the quantity inside the modulus of the left hand side of \eqref{mixedmix} as
    \begin{align}
        D:&=\int\Big(\prod_{j\in I}\varphi\circ f^{n_j-n_I}\Big)\diff\mu-\prod_{J\in\Qc}\Big(\int\Big(\prod_{j\in I\cap J}\varphi\circ f^{n_j-n_{I\cap J}}\Big)\diff\mu\Big) \nonumber \\
        &=\int\Big(\prod_{J\in\Qc}\Big(\prod_{j\in I\cap J}\varphi\circ f^{n_j-n_{I\cap J}}\Big)\circ f^{n_{I\cap J}-n_I}\Big)\diff\mu  
        -\prod_{J\in\Qc}\Big(\int\Big(\prod_{j\in I\cap J}\varphi\circ f^{n_j-n_{I\cap J}}\Big)\diff\mu\Big), \label{mixed-mixing}
    \end{align}
where for $L\subseteq [t]$, we set $n_L:=\min\{n_j\mid j\in L\}$. Note that all the iteration powers of $f$ appearing in \eqref{mixed-mixing} are non-negative.

Observe that the inequality \eqref{mixedmix} is homogeneous in $\varphi$. Hence, up to rescaling, we may assume $\|\varphi\|_\cA=1$. Now, for every $n\ge1$ we use the controlled tail property to write $\varphi\circ f^n=\varphi_{n,1}^+-\varphi_{n,1}^-+\varphi_{n,2}=:\varphi_{n,1}-\varphi_{n,-1}+\varphi_{n,2}$ with:
\begin{enumerate}
    \item $\|\varphi_{n,\pm1}\|_\cB \le cM\|\varphi\circ f^n\|_\cA\le cMC_f^n$;
    \item $\|\varphi_{n,2}\|_{L^q(\mu)}\le c_qe^{-\alpha_qM}\|\varphi\circ f^n\|_\cA\le c_qe^{-\alpha_qM}C_f^n$ for every $q\in\mathbb N^*$,
\end{enumerate}
where $c$ and $c_q$ are as in Definition \ref{ctp} and the constant $M$ is independent of $n$ and will be chosen later. Indexing all the possible combinations of the $v_j$'s in the $\varphi_{j,v_j}$'s with the multi-index $\bv:=(v_j : j\in I)\in\{1,-1,2\}^I$, the last term in \eqref{mixed-mixing} can be expanded as
\begin{align*}
    \sum_{\mathbf{v}\in\{1,-1,2\}^I}\bigg(\int\Big(\prod_{J\in\Qc}\Big(\prod_{j\in I\cap J}&\sign(v_j)\cdot\varphi_{n_j-n_{I\cap J},v_j}\Big)\circ f^{n_{I\cap J}-n_I}\Big)\diff\mu\\
    &-\prod_{J\in\Qc}\Big(\int\Big(\prod_{j\in I\cap J}\sign(v_j)\cdot\varphi_{n_j-n_{I\cap J},v_j}\Big)\diff\mu\Big)\bigg).
\end{align*}
It follows that $|D|$ is bounded by the sum of 
       \begin{align*}
        D_1=\sum_{\mathbf{v}\in\{1,-1,2\}^I\setminus \{1,-1\}^I}  \bigg|\int\Big(\prod_{J\in\Qc}\Big(&\prod_{j\in I\cap J}\varphi_{n_j-n_{I\cap J},v_j}\Big)\circ  f^{n_{I\cap J}-n_I}\Big)\diff\mu\bigg|  \\
        &+\sum_{\mathbf{v}\in\{1,-1,2\}^I\setminus \{1,-1\}^I}\bigg|\prod_{J\in\Qc}\Big(\int\Big(\prod_{j\in I\cap J}\varphi_{n_j-n_{I\cap J},v_j}\Big)\diff\mu\Big)\bigg|
        \end{align*}
        and 
 $$D_2=\sum_{\mathbf{v}\in\{1,-1\}^I}\bigg| \int\Big(\prod_{J\in\Qc}\Big(\prod_{j\in I\cap J}\varphi_{n_j-n_{I\cap J},v_j}\Big)\circ f^{n_{I\cap J}-n_I}\Big)\diff\mu
        -\prod_{J\in\Qc}\Big(\int\Big(\prod_{j\in I\cap J}\varphi_{n_j-n_{I\cap J},v_j}\Big)\diff\mu\Big)\bigg|. $$
      
We now bound $D_1$ and $D_2$. Observe that the case of $D_1$ is new with respect to \cite{bjo-gor-clt}, as in that work there are no tail terms $\varphi_{n,2}$ which are small in the $L^q(\mu)$-norms.

\smallskip

\textbf{Bound for} $D_1$. Each term in both sums can be estimated in modulus using Hölder's inequality and the fact that, since $\mu$ is $f$-invariant, the $L^q(\mu)$-norms are $f^*$-invariant. We have less than $2\cdot 3^t$ terms. Setting ${q_\mathbf{v}}=\big|\{j\in I:\, v_j=2\}\big|$, the term corresponding to a given $\mathbf{v}$ contributes at most
    \begin{align*}
        \Big(\prod_{J\in\Qc}\prod_{j\in I\cap J}C_f^{n_j-n_{I\cap J}}\Big)&\Big(\prod_{J\in\Qc}\prod_{j\in I\cap J,v_j\not=2}c M\Big)c_{q_\mathbf{v}}^{q_{\mathbf{v}}}\exp\Big(-\sum_{J\in\Qc}\sum_{j\in I\cap J,v_j=2}\alpha_{q_\mathbf{v}}M\Big)  \\
        &\lesssim_t C_f^{|I|a}c^{|I|}M^{|I\cap\{j:\, v_j\not=2\}|}\exp\Big(-\sum_{J\in\Qc}\sum_{j\in I\cap J,v_j=2}\alpha_{q_\mathbf{v}}M\Big) \\
        &\lesssim_tC_f^{ta}M^{|I\cap\{j:\, v_j\not=2\}|}\exp\Big(-\sum_{J\in\Qc}\sum_{j\in I\cap J,v_j=2}\alpha_{q_\mathbf{v}}M\Big).
    \end{align*}
The first inequality is due to the assumption $\max\Qc\le a$. Therefore, we have
\begin{equation} \label{dc1bound}
D_1 \lesssim_t \sum_{\mathbf{v}\in\{1,-1,2\}^I\setminus \{1,-1\}^I} C_f^{ta}M^{|I\cap\{j:\, v_j\not=2\}|}\exp\Big(-\sum_{J\in\Qc}\sum_{j\in I\cap J,v_j=2}\alpha_{q_\mathbf{v}}M\Big).
\end{equation}

\smallskip

\textbf{Bound for} $D_2$. We will use here the conditions (B2) and (B3). For all $n\in\mathbb{N}$, define $g_{n,\pm1}:=\varphi_{n,\pm1}/(cMC_f^n)$. Then $\|g_{n,\pm1}\|_\oB\le 1$, and so, up to the multiplicative constant
$$\prod_{J\in\Qc}\prod_{j\in I\cap J}c MC_f^{n_j-n_{I\cap J}}\lesssim_t c^{|I|}M^{|I|}C_f^{|I|a}\lesssim_t M^{|I|}C_f^{ta}$$
(the first inequality is due to the assumption $\max\Qc\le a$), we are left with the quantity
    \begin{equation}
        \sum_{\mathbf{v}\in\{1,-1\}^I}\bigg|\int\Big(\prod_{J\in\Qc}\Big(\prod_{j\in I\cap J}g_{n_j-n_{I\cap J},v_j}\Big)\circ f^{n_{I\cap J}-n_I}\Big)\diff\mu-\prod_{J\in\Qc}\Big(\int\Big(\prod_{j\in I\cap J}g_{n_j-n_{I\cap J},v_j}\Big)\diff\mu\Big)\bigg| \label{mixed-mixing-B}
    \end{equation}
to be estimated. Define the set of multi-indexes
$$W:=\big\{\mathbf{w}=(w_J: \, J\in\Qc)\in\N^{\Qc}\mid \, 1\le w_J\le m_{|I\cap J|}\,\text{ for every }\,J\in\Qc \big\}.$$
With the help of the functions $P_{t,l}$'s and with $m_t$ as in conditions (B3), we can rewrite 
$$ \prod_{j\in I\cap J}g_{n_j-n_{I\cap J},v_j}=\sum_{l=1}^{m_{|I\cap J|}} P_{|I\cap J|,l}(g_{n_j-n_{I\cap J},v_j}:\,j \in I\cap J). $$
Therefore, the expression in \eqref{mixed-mixing-B} is equal to
    \begin{align*}
        \sum_{\mathbf{v}\in\{1,-1\}^I} \bigg|\int\Big(\prod_{J\in\Qc}\sum_{l=1}^{m_{|I\cap J|}} P_{|I\cap J|,l}&(g_{n_j-n_{I\cap J},v_j}:\,j\in I\cap J)\circ f^{n_{I\cap J}-n_I}\Big)\diff\mu\\
        &-\prod_{J\in\Qc}\Big(\int\Big(\sum_{l=1}^{m_{|I\cap J|}} P_{|I\cap J|,l}(g_{n_j-n_{I\cap J},v_j}:\,j\in I\cap J)\Big)\diff\mu\Big)\bigg|\\
        \le \sum_{\mathbf{v}\in\{1,-1\}^I} \sum_{\mathbf{w}\in W}\bigg|\int\Big(\prod_{J\in\Qc}&P_{|I\cap J|,w_J}(g_{n_j-n_{I\cap J},v_j}:\,j\in I\cap J)\circ f^{n_{I\cap J}-n_I}\Big)\diff\mu\\
        &-\prod_{J\in\Qc}\Big(\int\Big(P_{|I\cap J|,w_J}(g_{n_j-n_{I\cap J},v_j}:\,j\in I\cap J)\Big)\diff\mu\Big)\bigg|.
    \end{align*}
Observe that there are at most $2^t\cdot m_t^t$ terms in the last sum, and for each of them we have
$$\big\|\pm P_{|I\cap J|,w_J}(g_{n_j-n_{I\cap J},v_j}:\,j\in I\cap J)\big\|_\oB\le\tilde{c}_{|I\cap J|}\le \tilde{c}_t\text{ for every }J\in\Qc.$$
Hence, we can apply the exponential mixing of order $|\Qc|$ for observables in $\cB$ to the functions $P_{|I\cap J|,w_J}(g_{n_j-n_{I\cap J},v_j}:\,j\in I\cap J)$'s (up to signs, see (B3)) for every $\mathbf{v}\in\{1,-1\}^I$ and $\mathbf{w}\in W$ and deduce, thanks to the assumption $\min\Qc>b$, that \eqref{mixed-mixing-B} is bounded by  a multiplicative constant (depending only on $t$) times 
$$\theta_t^{\min_{J_1,J_2\in\Qc}\{|n_{I\cap J_1}-n_{I\cap J_2}|\}}<\theta_t^b.$$
It follows that
\begin{equation} \label{dc2bound}
    D_2 \lesssim_t M^{|I|}C_f^{ta}\theta_t^b.
\end{equation}

\smallskip
    
Combining \eqref{dc1bound} and \eqref{dc2bound}, we have the following estimate for the modulus of $D$:
    \begin{equation}
        |D| \lesssim M^{|I|}C_f^{ta}\theta_t^b+\sum_{\mathbf{v}\in\{1,-1,2\}^I\setminus\{1,-1\}^I}\bigg(M^{|I\cap\{j:\, v_j\not=2\}|}C_f^{ta}\exp\Big(-\sum_{J\in\Qc}\sum_{j\in I\cap J,\\v_j=2}\alpha_{q_\mathbf{v}}M\Big)\bigg), \label{almost-there}
    \end{equation}
where the implicit constant depends only on $t$.

\medskip

Now we choose $$M:=-b\alpha_t^{-1}\log{\theta_t},$$
and set $C:=\max\{1,-\alpha_t^{-1}\log{\theta_t}\}$. Recall that the $\alpha_t$'s are non-increasing. Then, \eqref{almost-there} gives
 \begin{align*}
     |D| &\lesssim C^{|I|}C_f^{ta}b^{|I|}\theta_t^b+\sum_{\mathbf{v}\in\{1,-1,2\}^I\setminus\{1,-1\}^I}C^{|I\cap\{j:\, v_j\not=2\}|}C_f^{ta}b^{|I\cap\{j:\, v_j\not=2\}|}\theta_t^{b|\{j:\, v_j=2\}|}\\
     &\lesssim_t C_f^{ta}\max\{b^t,1\}\theta_t^b.
     \end{align*}
If $b\le1$, we take $\eta_t:=\theta_t$. If $b>1$, there exists $\theta_t<\eta_t<1$ such that $b^t\theta_t^b\lesssim\eta_t^b$, where the implicit constant depends only on $t$. In both cases, the proof is complete.
\end{proof}

\begin{remark}\rm
As in \cite{bian-dinh-sigma}, our proofs of Theorems \ref{thm-mixing} and \ref{thm-clt} can be easily adapted in the more general setting of invertible horizontal-like maps in any dimension (see \cite{dns-horizontal,ds-geometry}, and \cite{duj} for the case of dimension 2), under the natural assumption that the main dynamical degree dominates the other degrees \cite{bdr}. In this setting, the unique measure of maximal entropy is moderate by \cite{bdr-h-l}. On the other hand, our proofs cannot be directly adapted in the setting of automorphisms of compact Kahler manifolds, as non-trivial psh functions do not exist in this setting. As the theory of super-potentials is needed \cite{bd-kahler,ds-superpot-kahler,wu-kahler}, we postpone this study to a later work.
\end{remark}

\bigskip

\noindent\textbf{Acknowledgements.}
The first author is part of the PHC Galileo project G24-123. He would also like to thank his advisor Fabrizio Bianchi for useful discussions. This paper was partially written during the visit of the second author to Chinese Academy of Sciences. He would like to thank Song-Yan Xie and Zhangchi Chen for their warm hospitality.

\bigskip

\noindent\textbf{Funding.}
Hao Wu is supported by the Singapore MOE grant  MOE-T2EP20121-0013.


\end{document}